\documentclass{article}

\usepackage{amsmath,amsthm,amssymb,amsfonts}
\usepackage{graphicx}
\usepackage[english]{babel}
\usepackage[latin1]{inputenc}
\usepackage{url} 
\usepackage{bbm}

\newtheorem{theorem}{\large Theorem}
\newtheorem{proposition}[theorem]   {\large Proposition}

\newtheorem{corollary}[theorem]{\large  Corollary}

\font\tenmath=msbm10

\newfam\mathfam \textfont\mathfam=\tenmath

\def \\ { \cr }

\newcommand{\EE}{{\mathbb E}}
\newcommand{\ZZ}{{\mathbb Z}}

\newcommand{\PP}{{\mathbb P}}
\newcommand{\RR}{{\mathbb R}}
\newcommand{\A}{{\cal A}}

\newcommand{\I}{{\cal I}}
\newcommand{\K}{{\cal K}}

\newcommand{\N}{{\cal N}}

\newcommand\hatf{{\widehat f}}

\newcommand\hatV{{\widehat V}}

\newcommand\hatSigma{\widehat\Sigma}
\newcommand\hateta{{\widehat\eta}}
\newcommand\og{{\overline \Gamma}}

\newcommand\tildePP{{\widetilde\PP}}

\newcommand\hatVar{\mathop{\widehat{\rm Var}}}

\newcommand\ind{{\mathbbm{1}}}

\newcommand\inprobabilityto[1]{\xrightarrow[{#1}\to\infty]{\PP}}
\newcommand\indistributionto[1]{\xrightarrow[{#1}\to\infty]{d}}
\newcommand\ml{\left\{\begin{array}}
\newcommand\mr{\end{array}\right\}}
\newcommand\abs[1]{\left\vert#1\right\vert}

\newcommand\ds{\displaystyle}

\begin{document}

\title{A Gibbs approach to Chargaff's second parity rule}
\author{Andrew Hart%
\thanks{E-mail address:  \texttt{ahart@dim.uchile.cl}}
\qquad
Servet Mart\'{\i}nez%
\thanks{E-mail address:  \texttt{smartine@dim.uchile.cl}}
\qquad
Felipe Olmos%
\thanks{E-mail address:  \texttt{felipe@olmos.cl}}
\\
\ \\
\parbox{5in}{%
{\rmfamily\upshape\normalsize
${}^{*\dagger\ddagger}$ Departamento Ingenier{\'\i}a Matem\'atica and
Centro 
Modelamiento Matem\'atico, \newline
UMI 2071 CNRS-UCHILE,
Facultad de Ciencias F\'isicas y Matem\'aticas, \newline
Universidad de Chile,
Casilla 170, Correo 3, Santiago, Chile.}}
}

\date{\today}
\maketitle

\begin{abstract}
Chargaff's second parity rule (CSPR) asserts that the frequencies
of short polynucleotide chains are the same as those of the complementary reversed
chains. Up to now, this hypothesis has only been observed empirically and there
is currently no explanation for its
presence in DNA strands. Here we argue that CSPR is
a probabilistic consequence of the reverse complementarity between
paired strands, because the Gibbs distribution associated with
the chemical energy between the bonds satisfies CSPR.
We develop a statistical test to study the validity of CSPR under the
Gibbsian assumption and we
apply it to a large set of bacterial genomes taken from
the GenBank repository.  
\end{abstract}

\bigskip

\noindent {\footnotesize

\bigskip

}

\bigskip

\renewcommand\baselinestretch{1.5}
\footnotesize\normalsize

\section{Introduction}

Double helical DNA is made up of two complementary polynucleotide chains, the 
primary and the secondary strands, each having opposing polarities. Chargaff's 
first parity rule is that ``the numbers of $A$'s and $T$'s and the numbers of 
$C$'s and $G$'s match exactly in every DNA duplex'' \cite{Chargaff}. Chargaff's 
second parity rule (CSPR) states that this is valid when looking at a single 
strand, see \cite{Chargaff2}, and that this happens not only for mononucleotides 
but also for short polynucleotide chains.  Chargaff's first parity rule is a 
simple consequence of the double-stranded organization of genomic sequences and 
the chemistry of nucleic acids which only permits~$A$ to bond with~$T$ and~$C$ 
to bond with~$G$ In this work we argue that the reverse complementary 
relationship between nucleic acids on opposing strands also explains Chargaff's 
second parity rule when we assume that randomness manifest in DNA sequences is 
captured by the Gibbs distribution of the chemical energy potential.

CSPR was first observed experimentally in \emph{Bacillus
subtilis}~\cite{Chargaff2} and was subsequently confirmed in sufficiently long
sequences available in GenBank for small polymer chains of 3 to 6
bases~\cite{prabhu1993}. More recent empirical studies assessing its validity can
be found in~\cite{mitchell,PNAS_Albrecht} and \cite{hart&martinez2011}.

A number of possible mechanisms explaining strand symmetry have been proposed, 
for example, no strand biases for mutation and 
selection~\cite{lobry1995,sueoka1995} and selection of step-loop 
structures~\cite{forsdyke&bell2004}. Further discussion of various mechanisms 
that could support the origins of this intrastrand symmetry are discussed 
in~\cite{zhang&huang2010} and references therein. In~\cite{bell&forsdyke1999} 
and~\cite{Powdel&etal2009}, a number of mechanisms causing violation of CSPR in 
short polymers are described.

Here we propose that CSPR manifests due to the constraints that reverse 
complementarity imposes on the Gibbs distribution. In Section 
\ref{sec:framework} we give the framework and state the results.  There, the 
impirical polymer frequencies are replaced by polymer occurrence probabilities 
on a translation invariant probability distribution.  We then express CSPR using 
this notion and prove that CSPR written in this way follows from the fact that 
energy symmetry is preserved for the Gibbsian distribution.  This is done in 
Theorem~\ref{symmetrytheo}.  In Section \ref{sec:din} we give a characterisation 
of CSPR for dinucleotides, we prove an extension of the Central Limit Theorem 
for Gibbs measures to vector-valued random variables and we derive an explicit 
expression for the asymptotic covariance matrix. In Section \ref{sec:testgibbs} 
we supply a statistical test for the validity of CSPR for dinucleotides under 
the hypothesis that the nucleotides of the strand are distributed as a 
stationary Gibbsian process. We have applied the test extensively to bacterial 
genomes available from GenBank. The hypothesis of CSPR in the Gibbsian setting 
is confirmed for a large number of genomes.  Further analysis would be necessary 
in order to determine whether genomes rejected by the test were because 
they fail to comply with CSPR, because they are not Gibbsian, or both.

\section{Chargaff's second parity rule}
\label{sec:framework}

\subsection{Preliminaries}

Let $\A$ be a finite set (alphabet) endowed with an involution
$\Gamma:\A\to \A$: $\Gamma$ is one-to-one and $\Gamma^{-1}=\Gamma$.
In the genomic setting $\A=\{A,C,G,T\}$
and $\Gamma$ is an involution given by the complementary function
$\Gamma(A)=T$ and $\Gamma(C)=G$.

Let $x=(x_j: j=0,\ldots,n-1)\in \A^n$ be the sequence of nucleotides
on a strand of the genome (for bacterial DNA $n\approx 10^6 $).
The sequence $x$ complies with CSPR whenever 
the frequencies of all short polymers agree with the frequencies of their reverse complements.  In other words,
for $k$ small (order of $10)$ and all polymers $(a_0,\ldots, a_{k-1})\in \A^k$:
\begin{equation}
\label{eqhmo1}
\frac{\#\left\{j\! \leq \! n\!-\!k:
x_j\!=\!a_0,\ldots, x_{j+k-1}\!=\!a_{k-1})\right\}}{n-k+1}
=\frac{\#\left\{j\!\leq \!n\!-\!k:
x_j\!=\!\Gamma(a_{k-1}),\ldots, x_{j+k-1}\!=\!\Gamma(a_0)\right\}}
{n-k+1}.
\end{equation}
(Here $\#B$ denotes the cardinality of the set $B$). 
Observe that the frequency is computed by moving a window of length $k$
along the strand.

\subsection{CSPR as a symmetric probability relation}

We will derive CSPR from the complementary relation in
the thermodynamical formalism. In this theoretical framework
the strands are modeled by bi-infinite sequences and the frequencies
of a word are the probabilities that they appear at an arbitrary
place. We restrict ourselves to translation invariant
probability distributions that are
Gibbs measures with respect to the chemical energy.

The strands are modeled by sequences in $\A^\ZZ$. Thus, $x=(x_j: j\in \ZZ)$
represents the primary strand in the sense~$5'$ to~$3'$
while $y=(y_j: j\in \ZZ)$ represents the complementary strand in the sense~$3'$ to~$5'$.
They are related by reverse complementarity:  $y_j=\Gamma(x_{-j})$ for $j\in \ZZ$.
Let us write this rule in another way. Let
$\I:\A^\ZZ\to \A^\ZZ$ be the space reversal
involution given by $(\I(x))_j=x_{-j}$ and let
$\og:\A^\ZZ\to \A^\ZZ$ be such that $(\og(x))_j=\Gamma(x_j)$,
for $x\in\A^\ZZ$, $j\in \ZZ$.
Then, the rule of reverse complementarity may be written as $y=\og \circ \I(x)$.

A genome duplex is the pair $(x,y)$ and we denote by ${\widetilde \Psi}(x,y)$
its chemical energy, which  results from the interactions between the
nucleotides on both strands. Since the interactions
between the nucleotides are symmetric we assert that
\begin{equation}
\label{eqaragsym1}
{\widetilde \Psi}(x,y)={\widetilde \Psi}(y,x)\,.
\end{equation}
Insight into this equality may be obtained from the discussion of energy on finite pieces which appears
in~\cite{bowen2008}. Let $\Psi_l(x[-l,l]; y[-l,l])$ be
the energy in the portion $[-l,l]=\{-l,\ldots,l\}$ of the duplex.
In analogy with \cite{bowen2008} Page $5$, this energy can be assumed
to be given by
\begin{eqnarray}
\nonumber
\Psi_l(x[-l,l]; y[-l,l])&=&
\label{approxenefin}
\!\!\sum_{-l\leq j\leq k\leq l} \!\!\!\!\psi^s(k-j; x_k,x_j)
+\!\!\sum_{-l\leq j\leq k\leq l} \!\!\!\! \psi^s(k-j; y_{-j},y_{-k})\\
&{}& +\frac{1}{2}\!\sum_{-l\leq j, k\leq l}(\psi^o(|k-j|; x_k,y_{-j})\,.
\end{eqnarray}
The first two summations are due to the interactions between
sites on the same strand while the last one expresses the interactions
between sites on opposite strands. The quantity $\psi^s(r; a,b)$ is
the interaction between the nucleotides $a,b$ at distance $r$ on the
same strand and $\psi^o(r; a,b)$ is the interaction
between the nucleotides $a,b$ in opposite strands such that the distance
from the site containing one to the site in front of the other is~$r$
(recall that $y_{-j}$ is in front of $x_j$, so the distance
from site~$k$ containing $x_k$ to the site~$j$, which is in front
of the site containing $y_{-j}$, is $|k-j|$).
The expression (\ref{approxenefin}) is clearly symmetric in~$x$ and~$y$.

Let us express the symmetry relation (\ref{eqaragsym1}) in another way.
Since $y=\og \circ \I(x)$, the energy can be simply expressed as 
$\Psi(x)={\widetilde \Psi}(x,\og \circ \I(x))$ and
the symmetric dependence ${\widetilde \Psi}(x,y)={\widetilde \Psi}(y,x)$
between the strands implies that $\Psi$ satisfies the invariance property
$$
\forall x\in \A^\ZZ\,: \;\;\; \Psi(x)=\Psi(\og \circ \I (x))\,;
\hbox{ or equivalently } \,\Psi=\Psi\circ \og \circ \I\,.
$$

Next, the set $\A^\ZZ$ is endowed with the product $\sigma$-algebra and let $T:\A^\ZZ\to 
\A^\ZZ$ be the translation operator given by $(T(x))_j=x_{j+1}$ for all $j\in \ZZ$.  
Let $\PP$ be a translation invariant distribution on $\A^\ZZ$, that is
$$
\PP(T^{-1}B)=\PP(B), \;\;\; \forall \mbox{ measurable } B\subseteq \A^\ZZ \,.
$$ 

In the spirit of (\ref{eqhmo1}), we say that it satisfies CSPR if
\begin{equation}
\label{eqcspr222}
\forall x\!\in \!\A^\ZZ\, \, \forall k\geq1\,, \, \forall \,
(a_0,\ldots, a_{k-1})\in \A^k\,:
\;\; \PP(x_0\!=\!a_0,\ldots, x_{k-1}\!=\!a_{k-1})=
\PP(x_0\!=\!\Gamma(a_{k-1}),\ldots, x_{k-1}\!=\!\Gamma(a_0))\,.
\end{equation}

We claim that if $\PP$ is a translation invariant distribution on $\A^\ZZ$,
then property (\ref{eqcspr222}) is equivalent to
$\PP$ being $\og \circ \I-$invariant, that is,
it satisfies $\PP((\og \circ \I)^{-1}B)=\PP(B)$
for all measurable subsets~$B$ of $\A^\ZZ$. Indeed, from the equality
\begin{equation}
\label{eqset}
\og \circ \I^{-1}\{x: x_j=a_j,\ldots, x_k=a_k\}=
\{x: x_{-k}=\Gamma(a_k),\ldots, x_{-j}=\Gamma(a_j)\}
\end{equation}
taken together with the translation invariance property, one can show that
if~$\PP$ is $\og \circ \I$-invariant then (\ref{eqcspr222}) holds.
Conversely, the same translation invariance property combined with equality
(\ref{eqset}) may be used to prove that (\ref{eqcspr222}) implies
$\PP_\Psi((\og \circ \I)^{-1}B)=\PP(B)$ for
all cylinders~$B$.  Carath\'eodory's extension theorem then shows that this holds for
all measurable sets $B$ and the claim follows.

The following result derives compliance with CSPR from the symmetry of energy 
and will be proved in the next section, but in a more general framework than 
what is needed for dealing with genomic sequences. In the statement of the result 
we assume that the energy~$\Psi$ satisfies a $\theta$-H\"older property, which 
will be properly defined later in (\ref{holderdef}).

\begin{theorem}
\label{thm:chargaff222}
Assume $\Psi: \A^\ZZ\to \RR$ is $\theta$-H\"older
for some $\theta\in (0,1)$. Then, the
invariance $\Psi=\Psi\circ \og \circ \I$ implies that the
unique translation invariant Gibbs measure
$\PP_\Psi$ defined on $\A^\ZZ$ complies with CSPR, that is,
$\PP=\PP_\Psi$ satisfies condition (\ref{eqcspr222}).
\end{theorem}

\subsection{Gibbs measures and CSPR}

We employ a more general framework than the one 
stated in Theorem \ref{thm:chargaff222}.
Let $\A$ be a finite alphabet and
$\Xi=(\Xi(a,b): a,b\in \A)$ be an aperiodic $0-1$-valued matrix.
The shift of finite type defined by~$\Xi$ is the set
${\cal X}_{\Xi}= \left\{ x \in \A^\ZZ :
\Xi(x_j,x_{j+1}) = 1 \;  \forall j \in \ZZ \right\}$
endowed with the metric $\Delta_\theta(x,z) = \theta^{K(x,z)}$, where
$K(x,z)=\sup\{k\geq 0: x_i=z_i\;\, \forall |i| \leq k\}$ and $\theta\in (0,1)$
is an arbitrary but fixed value.  This metric induces the product topology.

Let $\theta \in (0,1)$ be fixed. Consider the set of
H\"older (continuous) functions in $({\cal X}_{\Xi},\Delta_\theta)$,
\begin{equation}
\label{holderdef}
F_{\theta}=\left\{g\in C({\cal X}_{\Xi}) : |g|_\theta < \infty \right\}
\hbox{ where } |g|_{\theta}  = 
\sup\left\{\frac{|g(x)-g(z)|}{\Delta_\theta(x,z)}: x,y\in {\cal X}_{\Xi}, x\neq z\right\}\,.
\end{equation}
The linear set $F_{\theta}$ is a Banach space when it is endowed with the norm
$\|g \|_{\theta} = \|g\|_{\infty} + |g|_{\theta}$.

Each Gibbs measure on ${\cal X}_\Xi$ is 
defined by an energy function $\Psi\in F_\theta$.  The value $\Psi(x)$
represents the energy of the system in state $x \in {\cal X}_\Xi$.
In the thermodynamic formalism of shifts of finite type, it has been shown
(see Theorem~1.2 in \cite{bowen2008}, Pages 5--6)
that there exists a unique translation invariant probability
measure $\PP_\Psi$ that satisfies
\begin{equation}
\label{bowen11}
\exists \, 0<c_1<c_2<\infty\,, p\in \RR \,, \;
\forall z\in {\cal X}_\Xi\,, \forall \, k\geq 0\,: \;\;
c_1 \leq \frac{\PP_\Psi \left(x: x_0 = z_0 , \ldots , x_k = z_k\right)}
{e^{-pk + \sum_{i=0}^{k-1} \Psi(T^i z)}} \leq c_2\,.
\end{equation}
A detailed proof of this result as well as a complete
exposition of this topic is given in \cite{bowen2008}, Pages 3--16.
In this reference, it is also proven that the constant
$p=p(\Psi)$ is the pressure of $\Psi$
and that the probability measure $\PP_\Psi$ is the unique
translation invariant probability measure satisfying the variational
principle $p(\Psi)=h_{\PP_\Psi}(T)+\int \Psi d\PP_\Psi$, where $h_{\PP_\Psi}(T)$
is the entropy of $T$ for the translation invariant distribution $\PP_\Psi$.

From now on we assume that the aperiodic matrix $\Xi$ also satisfies
\begin{equation}
\label{symmmatr}
\forall \, a,b\in \A\,: \;\;\; \Xi(a,b)=\Xi(\Gamma(b),\Gamma(a))\,.
\end{equation}
We note that $\Xi(a,b)=1$ for all $a,b\in \A$ in the genomic framework,
so in this context $\Xi$ always satisfies condition (\ref{symmmatr}). Hence,
Theorem~\ref{thm:chargaff222} is a straight forward
consequence of the following result.

\begin{theorem}
\label{symmetrytheo}
Assume $\Xi$ satisfies (\ref{symmmatr}).
Let $\Psi\in F_\theta$. Assume that $\Psi$ is
$\og \circ \I$-invariant: $\Psi(x)=\Psi(\og \circ \I (x))$ for all
$x \in {\cal X}_\Xi$. Then the unique translation invariant Gibbs probability measure 
$\PP_\Psi$ is $\og \circ \I$-invariant
and hence complies with CSPR:
$$
\forall \, k\geq 0\,, (z_0,\ldots, z_k)\in \A^{k+1}\,:\;\;\;
\PP_\Psi\left(x: x_0 = z_0 , \ldots , x_k = z_k \right)=
\PP_\Psi\left(x: x_0 = \Gamma(z_k) , \ldots , x_k = \Gamma(z_0) \right)\,.
$$
\end{theorem}

\begin{proof}
To begin, let $\PP=\PP_\Psi$ denote the unique $T$-invariant probability measure 
on ${\cal X}_\Xi$ that satisfies (\ref{bowen11}).  
Define the probability measure~$\tildePP$ as $\tildePP(B)=\PP((\og \circ \I)^{-1}B)$ 
for all measurable sets~$B$ in ${\cal X}_\Xi$.

\noindent Claim $1$: $\tildePP$ is translation invariant. This can be proved 
as follows. Note that $\I^{-1}=\I$ and $\og^{-1}=\og$, while $\og$ commutes with 
$\I$, $T$ and $T^{-1}$. So $(\og \circ \I)^{-1}=(\og \circ \I)$. We also have 
$\I\circ T^{-1}=T\circ \I$ and hence
$$
(\og \circ \I)^{-1}\circ T^{-1}=T\circ \og \circ \I\,.
$$
Since~$\PP$ is $T-$invariant, it is also $T^{-1}$-invariant, so
$$
\tildePP(T^{-1}(B))=\PP((\og \circ \I)^{-1} \circ T^{-1}(B))=
\PP(T\circ \og \circ \I(B))=
\PP(\og \circ \I(B))=\PP((\og \circ \I)^{-1}(B))=\tildePP(B)\,.
$$
which yields the claim.

\noindent Claim $2$: $\tildePP$ satisfies
$$
\exists \, 0<\widetilde{c_1}<\widetilde{c_2}<\infty \;
\forall z\in {\cal X}_\Xi\, \forall \, k\geq 0\,: \;\;
\widetilde{c_1} \leq \frac{\tildePP\left(x: x_0 = z_0, \ldots ,x_k = z_k \right)}
{e^{-pk + \sum_{i=0}^{k-1} \Psi(T^i z)}} \leq \widetilde{c_2}\,.
$$
Note that once this claim has been shown, the result will immediately follow 
because uniqueness of~$\tildePP$ implies $\PP=\tildePP$, and so~$\PP$ is $\og 
\circ \I$-invariant. To prove the claim, first observe that since $\Gamma^{-1}=\Gamma$ 
and $\PP$ is $T$-invariant,
\begin{eqnarray*}
\nonumber
&{}&
{\widetilde {\PP}}\left(x: x_0 = z_0, \ldots, x_k = z_k \right)=
\PP\left(x: (\og(\I(x)))_0=z_0, \ldots,
(\og(\I(x)))_k = z_k \right)\\
\nonumber
&{}&\; =\PP\left(x: \Gamma(x_0)=z_0, \ldots, \Gamma(x_{-k})=z_k \right)
=\PP\left(x: x_0=\Gamma(z_0),\ldots,x_{-k} =\Gamma(z_k) \right)
\\
\nonumber
&{}& \; =\PP\left(x: x_0=\Gamma(z_k), \ldots, x_{k} = \Gamma(z_0) \right)
=\PP\left(x: x_0=(\og \circ \I(z))_{-k}, \ldots ,x_{k}
= (\og \circ\I(z))_{0}\right)\\
&{}& \; =\PP\left(x: x_0=(T^{-k}(\og\circ \I(z)))_{0},\ldots,
x_{k} =(T^{-k}(\og \circ \I(z)))_{k}\right)\,.
\end{eqnarray*}

On the other hand, from the equality $T^{i-k} (\og \circ \I (z))= \og \circ 
\I(T^{k-i} (z))$ and using the fact that~$\Psi$ is $\og \circ \I$-invariant, we 
obtain
$$
\sum_{i=0}^{k-1} \Psi(T^iT^{-k}(\og \circ \I (T^{-1}z)))=
\sum_{i=0}^{k-1} \Psi(\og \circ \I( T^{k-i-1}z))=
\sum_{i=0}^{k-1} \Psi( T^{k-i-1}z)=\sum_{i=0}^{k-1} \Psi( T^{i}(z))\,.
$$
Hence
\begin{equation}
\label{eq334}
\frac{{\widetilde {\PP}}\left(x: x_0 = z_0, \ldots, x_k = z_k \right)}
{e^{-pk + \sum_{i=0}^{k-1} \Psi(T^i (z)}}=
\frac{\PP\left(x: x_0=(T^{-k}(\og \circ \I(z)))_{0},\ldots, x_{k}
=(T^{-k}(\og \circ\I(z)))_{k}\right)}
{e^{-pk + \sum_{i=0}^{k-1}\Psi(T^i T^{-k}(\og \circ \I (T^{-1}z)))}}\;.
\end{equation}
We note that
$$
\forall z\in {\cal X}_\Xi\, \forall \, k\geq 0\,: \;\;
{\tilde c}_1 \leq \frac{e^{-pk + \sum_{i=0}^{k-1} \Psi(T^i T^{-1}(z))}}
{e^{-pk + \sum_{i=0}^{k-1} \Psi(T^i (z))}}
\leq {\tilde c}_2\,,
$$
with ${\tilde c}_1=e^{\min \Psi-\max \Psi}$ and ${\tilde c}_2=e^{\max \Psi-\min 
\Psi}$.  Then,
\begin{equation}
\label{eq555}
{\tilde c}_1 \leq
\frac{e^{-pk + \sum_{i=0}^{k-1}  \Psi(T^iT^{-k}(\og \circ \I (T^{-1}z)))}}
{e^{-pk + \sum_{i=0}^{k-1}  \Psi(T^iT^{-k}(\og \circ \I (z)))}}
\leq {\tilde c}_2\,.
\end{equation}
Hence from (\ref{eq334}), (\ref{bowen11}) and (\ref{eq555}), we deduce
that Claim~$2$ holds,
$$
\forall z\in {\cal X}_\Xi\, \forall \, k\geq 0\,: \;\;
c_1{\tilde c}_1 \leq \frac{{\widetilde {\PP}}\left(x: x_0 = z_0, \ldots , x_k = z_k \right)}
{e^{-pk + \sum_{i=0}^{k-1} \Psi(T^i z)}} \leq c_2{\tilde c}_2\,.
$$
Hence, $\tildePP=\PP$ and the proof is complete.
\end{proof}

\section{CSPR for dinucleotides}
\label{sec:din}

\subsection{A $5$-dimensional characterisation of CSPR}

Henceforth, we shall focus on the dinucleotide distributions under CSPR.  
Let $\PP$ be a translation invariant distribution on $\A^\ZZ$.  
As stated, CSPR means that for all $R\geq 1$, we have
\begin{equation}
\label{secondrule1}
\forall \, (a_0,\ldots,a_{R-1})\in \A^R\,:
\;\;\;\, \PP(x: x_0=a_0,\ldots,x_{R-1}=a_{R-1})=
\PP(x: x_0=\Gamma(a_{R-1}),\ldots,x_{R-1}=\Gamma(a_0))\,.
\end{equation}
If the set of equalities (\ref{secondrule1}) holds for some
$R=R_0$,  we say that CSPR holds for $R_0$. In this case,
by taking appropriate marginals,
the equalities also hold for all positive integers $R\leq R_0$.

Now, for discussing CSPR for $R=2$,  it is convenient to introduce the following 
notation. Let $[ab]_k$ be the event $\left\{x: x_k = a,\, x_{k+1} = b \right\}$. 
Since $\PP$ is translation invariant, we have $\PP([ab]_k)=\PP([ab]_0)$ for all 
$k\in\ZZ$ and $a,b\in\A$.
Therefore, CSPR for $R=2$ reduces to
\begin{equation}
\label{eqn:chargaff.r2}
\forall \, a,b\in \A\,: \quad \PP([ab]_0) = \PP([\Gamma(b)\Gamma(a)]_0).
\end{equation}
This equality implies CSPR for $R=1$:  $\PP([a]_0) = \PP([\Gamma(a)]_0)$ for $a\in\A$, 
where $[a]_0=\{x: x_0=a\}$.

We want to test the hypothesis $H_0$: CSPR holds for $R=2$. In order to 
construct such a test, it is useful to introduce the following quantities:
\begin{equation}
\label{def.f}
f=(f(a,b): (a,b)\in \A^2) \; \hbox{ where } \;
f(a,b):=\PP([ab]_0)-\PP([\Gamma(b)\Gamma(a)]_0)\,.
\end{equation}
From (\ref{eqn:chargaff.r2}), 
CSPR for $R=2$ is satisfied if and only if $f=0$.

We remark that  $4$ of the above $16$ quantities $f(a,b)$ vanish. More 
precisely, whenever $(a,b)=(c,\Gamma(c))$ for some $c\in \A$, we see that 
$f(a,b)= 0$. Moreover, among the remaining $12$ terms, only $5$ are meaningful 
since $f(a,b)=- f(\Gamma(b),\Gamma(a))$ for any $a,b\in \A$, and 
$\sum_{c\in\A}f(a,c)=\sum_{c\in\A}f(c,a)$ for all $a\in\A$.
In the following, we fix an index set $\K=\{(A,A),(A,C),(A,G),(C,A),(C,C)\}$ for 
$5$ of these values and gather them together into a vector $f^\K 
:=\left(f(a,b):\ (a,b)\in \K\right)$. Using this alternative representation, the 
null hypothesis $H_0$ is satisfied if and only if $f^\K=0$.

\subsection{Covariances and the Central Limit Theorem}

Here, we present some results that are essential for developing an asymptotic 
test for the hypothesis $H_0$: CSPR holds for $R=2$, in the setting of Gibbs 
distributions.
 
 Let $\PP=\PP_\Psi$ be Gibbsian for some $\Psi \in F_{\theta}$, with $\theta\in 
 (0,1)$ fixed. We begin by giving a simple computation. Let $\EE=\EE_\Psi$ 
 denote the expectation operator associated with $\PP_\Psi$. A function $g\in 
 F_\theta$ is said to be of zero mean if $\EE(g)=0$.

In this section we assume $\varphi^1, \ldots, \varphi^l$ are zero mean functions in
$F_{\theta}$ and set $\varphi=(\varphi^1, \ldots, \varphi^l)$.  We shall consider
$X_i^k := \varphi^k \circ T^i$ for $i\geq 0$ and $k=1,\ldots, l$, and define for $n\geq1$,
\begin{equation}
\label{defS12}
S_n^k := \frac1{\sqrt n} \sum_{i=0}^{n-1} X_i^k\,.
\end{equation}

\begin{proposition}
\label{covsimp1}
The limits
\begin{eqnarray}
\nonumber
&{}& \Sigma^\varphi(k,j)=\lim\limits_{n\to \infty}\EE_\Psi(S_n^k \, S_n^j)
\,\mbox{ exist for all } k,j\in \{1,\ldots,l\} \,\mbox{ and }\\
\label{covexp2}
&{}& \Sigma^\varphi(k,j)=\EE_\Psi \left(X_0^k \, X_0^j \right)+\sum_{i=1}^{\infty}
\EE_\Psi \left( X_0^k \, X_i^j \right)
		+\sum_{i=1}^{\infty} \EE_\Psi \left( X_0^j \,  X_i^k \right)\,.
\end{eqnarray}
The matrix $\Sigma^\varphi=\left(\Sigma^\varphi(k,j):
k,j\in \{1,\ldots,l\}\right)$
is symmetric and semi-positive definite.

Moreover the convergence of the two summations on the right-hand side 
of~(\ref{covexp2}) occurs at a geometric rate, more precisely, 
\begin{equation}
\label{eqn:geom.rate}
\exists \, \bar\delta<\infty, \, \xi\in (0,1)\,,\; \forall k,j=1,\ldots,l\,,
\, \forall i\geq 1\,:\;\;\; 
\abs{\EE_\Psi(X_0^k X_i^j)} \leq \bar\delta\xi^i \,.
\end{equation}
\end{proposition}

\begin{proof}
By expanding the terms in the sum and
using the translation invariance property
$\EE(X_i^k \, X_r^j)=\EE(X_0^k \, X_{r-i}^j)$
for all $k,j\in \{1,\ldots,l\}$ and $i<r$, we get
\begin{eqnarray*}
\EE(S_n^k S_n^j) &=&
\frac1n \sum_{i=0}^{n-1} \EE \left( X_i^k \, X_i^j \right)+
\frac1n \sum_{i=1}^{n-1} \sum_{r=0}^{i-1} \EE \left( X_i^k \, X_r^j \right)
+\frac1n \sum_{i=1}^{n-1} \sum_{r=0}^{i-1} \EE \left( X_i^j \, X_r^k \right)
\\
&=&
 \EE \left( X_0^k \, X_0^j \right)+\frac1n
\sum_{i=1}^{n-1} (n-i)  \EE \left( X_0^k \, X_i^j \right)
+\frac1n \sum_{i=1}^{n-1} (n-i)  \EE \left( X_0^j \, X_i^k \right)\,.
\end{eqnarray*}
Since $\varphi^k \in F_\theta$ for each $k$, the exponential cluster
property of Gibbs measures (see Property~1.26 on Page~23 in~\cite{bowen2008})
guarantees the existence of $\delta<\infty$, and $\xi \in (0,1)$ only depending on $\theta$
and $\Psi$, such that for all $k,j\in \{1,\ldots,l\}$ ,
$$
\abs{\EE ( X_0^k \, X_i^j )}
= \abs{\EE \bigl(\varphi^k \, \cdot \, (\varphi^j \circ T^i) \bigr)} 
\leq \delta \|\varphi^k\|_\theta \|\varphi^j\|_\theta \xi^i\,.
$$
As a consequence, (\ref{eqn:geom.rate}) is satisfied. 
Hence all the series are absolutely convergent. Moreover, since
$\sum_{i=0}^\infty i\EE(X_0^k X_i^j)$ is finite, the Cesaro mean
of $i\EE(X_0^k X_i^j)$ converges to zero and we obtain the formula
$$
\lim_{n \to \infty}\EE(S_n^k \, S_n^j) = \EE \left(X_0^k\,X_0^j \right)
+ \sum_{i=1}^{\infty} \EE \left(X_0^k \,  X_i^j \right)+
\sum_{i=1}^{\infty} \EE \left(X_0^j \,  X_i^k \right) = \Sigma_{kj}^\varphi\,.
$$
Finally, we see from this explicit expression that the matrix $\Sigma^\varphi$ 
is symmetric and semi-positive definite because each matrix $\left(\EE(S_n^k 
S_n^j): k,j\in \{1,\ldots,l\}\right)$ is a  covariance matrix. Hence, its limit 
$\Sigma^\varphi$ is also semi-positive definite.
\end{proof}

Next, we show a Central Limit Theorem for random vectors
in the Gibbs framework, which is a corollary of the Central Limit Theorem
given in \cite{Coelho1990}.

\begin{proposition}
\label{prop:vector_clt}
If $\Sigma^\varphi=(\Sigma^\varphi(k,j): k,j=1,\ldots, l)$ given by (\ref{covexp2}) is
positive definite then the vector process $Z_n := (S_n^1, \ldots, S_n^l)$
(where $S_n^k$ is given in (\ref{defS12})) converges in
distribution to the multivariate normal vector
$\N(0,\Sigma^\varphi)$.
\end{proposition}

\begin{proof}
We recall that the Central Limit Theorem shown in \cite{Coelho1990}
says that if a function
$g \in F_{\theta}$ is of zero mean and
$$
\sigma(g)^2 \neq 0  \mbox{ where }  \sigma(g)^2:=\lim_{n \to \infty} \frac1n \;
\EE\left( \left( \sum_{i = 0}^{n-1} g \circ T^i  \right)^2 \right)\,,
$$
then
\begin{equation}
\label{coeparr}
\frac1{\sqrt n}\left(\sum_{i=0}^{n-1} g \circ T^i\right)
\indistributionto{n} \N(0, \sigma(g)^2)\,.
\end{equation}

Since $F_{\theta}$ is Banach, for all $\alpha = (\alpha_1, \ldots, \alpha_l)$
the function $\varphi_\alpha = \sum_{k=1}^l \alpha_k \varphi^k$ is in $F_{\theta}$.
Since $\varphi_\alpha$ has zero mean and $\sum_{i=0}^{n-1}(\varphi_\alpha\circ T^i) / \sqrt n =\alpha' Z_n$,
where $a'$ denotes the transpose of a vector~$a$, the Central Limit Theorem (\ref{coeparr}) gives
\begin{equation}
\label{convalpha}
\alpha' Z_n \indistributionto{n} \N(0, \sigma_{\alpha}^2),
\end{equation}
where
$$
\sigma_\alpha^2 =  \lim_{n \to \infty} \frac1n
 \EE \left( \left( \sum_{i = 0}^{n-1} \varphi_\alpha \circ T^i  \right)^2 \right) \,,
$$
provided that $\sigma_\alpha^2\neq0$.  Now, $\sigma_\alpha^2\neq0$ will be obtained as a consequence of the fact that
\begin{equation}
\label{eqshape}
\sigma_\alpha^2 =\alpha' \Sigma^\varphi \alpha =
\sum_{k = 1}^l \sum_{j = 1}^l \alpha_k \alpha_j \Sigma^\varphi(k,j)\,.
\end{equation}
This together with the assumption that $\Sigma^\varphi$ is positive
definite allows us to determine that $\sigma_\alpha^2=\alpha' \Sigma^\varphi \alpha > 0$
for all $\alpha\neq 0$.  Furthermore, (\ref{convalpha}) and (\ref{eqshape}) yield
$$
\forall s \in \RR\,:\;\;\;
\lim_{n \rightarrow \infty} \EE(e^{i s \alpha' Z_n})
=  e^{- \frac12 s^2 \sigma_\alpha^2 } = e^{-\frac12 s^2
\alpha' \Sigma^\varphi \alpha}\,,
$$
which is the characteristic function of a $\N(0, \alpha' \Sigma^\varphi \alpha)$
random vector. Convergence of $Z_n$ in
distribution to an $\N(0, \Sigma^\varphi)$ random vector 
then follows from L\'evy's continuity theorem.

It only remains to prove (\ref{eqshape}). Notice that
$$
\sum_{i = 0}^{n-1} \varphi_\alpha \circ T^i
   = \sum_{k = 1}^l \alpha_k \sum_{i = 0}^{n-1} \varphi^k \circ T^i
   =  \sum_{k = 1}^l \alpha_k \sum_{i = 0}^{n-1} X_n^k
   = \sqrt n \sum_{k = 1}^l \alpha_k S_n^k
   $$
which implies that
$$
\sigma_\alpha^2 = \lim_{n \rightarrow \infty}
 \EE\left( \left( \sum_{k = 1}^l \alpha_k S_n^k \right)^2 \right)
= \sum_{k = 1}^l \sum_{j=1}^l
\alpha_k \alpha_j \lim_{n \to \infty} \EE(S_n^k S_n^j)\,.
$$
Finally, Proposition~\ref{covsimp1} asserts that $\lim_{n \to \infty} \EE(S_n^k S_n^j) 
= \Sigma^\varphi(k,j)$ and hence the result follows.
\end{proof}

\section{Testing under the Gibbsian assumption}
\label{sec:testgibbs}

\subsection{A statistical test}

Recall that the hypothesis $H_0$: CSPR for $R=2$, is equivalent to $f=0$, 
where~$f$ was defined in~(\ref{def.f}).  
Let us introduce estimators of the various quantities involved in testing this. For any finite observed sequence
$X=(X_0,\ldots,X_{n-1})$, let
\begin{equation*}
\hatf_n(a,b) := \frac{N_n(a,b)}{n} -\frac{N_n(\Gamma(b),\Gamma(a))}{n},
\end{equation*}
where
$$
N_n(a,b) := \#\{k\in \{0,\ldots,n-1\}: \, (X_k,X_{k+1})=(a,b)\}
$$
counts the number of occurrences of the pattern $ab$ in the sequence.  
Note that we treat the sequence $X$ as though it
were circular with $X_{n-1}$ connected to $X_0$, so that $X_n\equiv X_0$.

We shall show that one appropriate statistic  for assessing this test is
$$
\hateta_n =n \, \hatf^\K_n\,{}'\, \hatV_n ^{-1} \,\hatf^\K_n \,,
$$
where $\hatf_n^\K =\left(\hatf_n(a,b):\ (a,b)\in \K\right)$ is a consistent unbiased
estimator of $f^\K$ and $\hatV_n$ is a consistent biased estimator of the
asymptotic covariance matrix~$V$ of $\sqrt n\hatf_n^\K$ which we shall define shortly.  
Furthermore, we shall prove
that $\hateta_n$ converges asymptotically
in distribution to a $\chi_5^2$ random variable.  Then, sufficiently large
values of $\hat\eta_n$ will identify sequences that fail to comply with CSPR for $R=2$.

More precisely, the test is set up as follows:
\begin{equation*}
  \text{Reject } H_0 \text{ if } \hateta_n \geq s,
\end{equation*}
where $s$ is some threshold to be chosen. If~$\alpha$ is the type~I error 
desired for the test (for instance $\alpha =0.05$ or $0.01$), then we require 
that
$$
\PP_{H_0}(\text{reject } H_0) =\PP_{H_0}(\hateta_n \geq s) \leq \alpha,
$$
either exactly or asymptotically.  Doing this exactly is not feasible in the
current setting, but $\hat\eta_n\indistributionto{n} \chi_5^2$, where $\indistributionto{n}$ denotes convergence in distribution.  Thus the
threshold~$s$ can be fixed asymptotically by appealing to the $\chi^2$
distribution on $5$ degrees of freedom.  We merely have to set~$s$ to the $1- \alpha$
quantile $\chi_{5,1- \alpha}^2$ of the $\chi_5^2$ distribution. 

\subsection{Asymptotics of the test statistic}

In order to construct this asymptotic test, we make the further assumption that 
the distribution $\PP=\PP_\Psi$ is Gibbsian for some energy $\Psi \in 
F_{\theta}$, where $\theta\in (0,1)$.  Recall that $\PP$ is ergodic. Let~$\EE$ 
denote the mean expected value operator associated with~$\PP$.

 Firstly,
 $$
 \EE(\hatf_n^\K)
 = \EE\left( \frac{N_n(a,b)}n - \frac{N_n(\Gamma(b),\Gamma(a))}n\right)
 = \frac1n \left(n\PP([ab]_0)-n\PP([\Gamma(b)\Gamma(a)]_0)\right)
 = f(a,b).
$$
From ergodicity the law of large numbers holds and so
$$
\lim\limits_{n\to \infty} \frac{N_n(a,b)}n =\PP([ab]_0)\;\;\, \PP-\hbox{a.e.}
\mbox{ and hence } \lim\limits_{n\to \infty}\hatf_n^\K =f^\K \,\;\; \PP-\hbox{a.e.}
$$
Therefore $\hatf_n^\K$ is a consistent, unbiased estimator of $f^\K$.

Next define
$$
\varphi=(\varphi^{a,b}: (a,b)\in \A^2) \hbox{ where }
\varphi^{a,b}:=\ind_{[ab]_0} - \PP([ab]_0) \,,
$$
where, as usual, $\ind_B$ is the characteristic function of the set $B$.
Observe that for all $(a,b)\in \A^2$ we have  $\varphi^{a,b}\in F_\theta$.
For $i\geq 0$ and $n\geq1$, define
$$
\forall \, (a,b) \in \A^2\,:\;\; X_i^{a,b}=\varphi^{a,b}\circ T^i\,
\hbox{ and } \,
S_n^{a,b}=\frac{1}{\sqrt n} \sum_{i = 0}^{n-1} X_i^{a,b} \,.
$$
A simple calculation that takes advantage of the $T$-invariance of~$\PP$ gives
$$
X_i^{a,b}=(\ind_{[ab]_i} - \PP([ab]_0)) \, \hbox{ and }
S_n^{a,b}=\frac{1}{\sqrt n}\left(N_n(a,b) - n\PP([ab]_0)\right)\,.
$$
A straight forward application of Proposition~\ref{covsimp1} can be used to show 
existence of the matrix $\Sigma^\varphi=(\Sigma^\varphi(a,b;c,d): (a,b), 
(c,d)\in\A^2)$, whose elements are defined by
$$
\Sigma^\varphi(a,b; c,d):=
\lim_{n\to \infty}\EE(S_n^{a,b} \, S_n^{c,d})\,.
$$
(Note that for simplicity we write $\Sigma^\varphi(a,b;c,d)$ rather than 
$\Sigma^\varphi((a,b),(c,d))$). Furthermore, using (\ref{covexp2}) from the same 
lemma, we can see that
$$
\Sigma^\varphi(a,b;c,d)=\EE\left(X_0^{a,b} \, X_0^{c,d} \right)+
\sum_{i=1}^{\infty} \EE\left( X_0^{a,b} \, X_i^{c,d} \right)
+\sum_{i=1}^{\infty} \EE\left( X_0^{c,d} \,  X_i^{a,b} \right)\,,
$$
and that $\Sigma^\varphi$ is symmetric and semi-positive definite.

Further simple  computations enable us to write the elements of $\Sigma^\varphi$ 
explicitly as
\begin{eqnarray*}
\Sigma^\varphi(a,b;c,d)
&=& \PP([ab]_0\cap[cd]_0) - \PP([ab]_0)\PP([cd]_0)
 + \sum_{k=1}^{\infty} \left[ \PP([ab]_0 \cap [cd]_{k}) - \PP([ab]_0)
\PP([cd]_0)  \right] \\
&{}& \, +\sum_{k = 1}^{\infty} \left[\PP([cd]_0 \cap [ab]_{k})-\PP([ab]_0)
\PP([cd]_0) \right]\,.
\end{eqnarray*}

As a corollary to Proposition~\ref{prop:vector_clt}, we obtain:

\begin{proposition}
Assume $\Sigma^\varphi$ is positive definite. Then, the joint distribution of the
counts $N_n(a,b)$ asymptotically satisfy
$$
\left(\frac{N_n(a,b) - n\PP([ab]_0)}{\sqrt n}: (a,b) \in \A^2\right)
\indistributionto{n} \N(0, \Sigma^\varphi)\,.
$$
\end{proposition}

The following is then obtained by taking appropriate marginals in the preceding result.

\begin{corollary}
Assume $\Sigma^\varphi$ is positive definite.  We have
$$
\sqrt n(\hatf_n^\K -f^\K) \indistributionto{n} \N(0, V),
$$
where the covariance matrix $V=\left(V(a,b;c,d):
(a,b),(c,d)\in\K\right)$ is given by
$$
V(a,b;c,d)=
\Sigma^\varphi(a,b;c,d) +
\Sigma^\varphi(\Gamma(b), \Gamma(a); \Gamma(d), \Gamma(c)) -
\Sigma^\varphi(\Gamma(b), \Gamma(a); c,d) -
\Sigma^\varphi(a,b; \Gamma(d),\Gamma(c))\,.
$$
\end{corollary}

From this result we conclude under the hypothesis $H_0: f^\K=0$ that
$$
\sqrt n \,\hatf_n{}^\K \indistributionto{n} \N(0,V).
$$
As a consequence, $n\hatf_n^{\K\, \prime}
V^{-1}\hatf_n{}^\K$ converges in distribution to a $\chi^2$ distribution
on~$5$ degrees of freedom, provided that $V$ is positive definite.

Observe that $\hatf_n^\K = \frac1n\lambda N_n$, where $N_n=\left(N_n(a,b): (a,b)\in\A^2\right)$ and
$\Lambda=\left(\Lambda(a,b;c,d): (a,b)\in\K, (c,d)\in\A^2\right)$ is the $5\times16$ 
matrix given by
$$
\Lambda(a,b;c,d) := \left\{\begin{array}{ll}
1, & \mbox{if } (a,b)=(c,d), \\
-1, & \mbox{if } (a, b) = (\Gamma(d),\Gamma(c)), \\
0, & \mbox{otherwise.}
\end{array}\right.
$$
The covariance matrix~$V$ may then be written as $V=\Lambda\Sigma^\varphi \Lambda'$. 
Since $\Lambda$ is of full rank, $V$ is positive definite 
whenever~$\Sigma^\varphi$ is positive definite. 

\begin{proposition}
\label{prop:testasym}
Assume that $\Sigma^\varphi$ is positive definite. Then, there exists a 
consistent estimator $\hatV_n$ of $V$
such that $\hateta_n:=n \hatf_n^\K{}' \hatV_n^{-1} \hatf_n^\K$ converges in
distribution to a $\chi_5^2$ random variable.
\end{proposition}

The proof of this proposition will be a consequence of the following
constructions and intermediate results. 

In order to define the estimator $\hatV_n$ of the covariance matrix~$V$, we 
first require an estimator of~$\Sigma^\varphi$.  Let $\hatSigma_{n,m}= 
\left(\hatSigma_{n,m}(a,b;c,d): (a,b),(c,d)\in\A^2\right)$, where
\begin{eqnarray}
\nonumber
\hatSigma_{n,m}(a,b;c,d)&:=& 
\frac{N_n^{(0)}(a,b;c,d)}n - \frac{N_n(a,b)}n \cdot \frac{N_n(c,d)}n\\
\nonumber
&{}& \, + \sum_{i=1}^m \left( \frac{N_n^{(i)}(a,b;c,d)}n - 
\frac{N_n(a,b)}n \cdot \frac{N_n(c,d)}n \right) \\
\label{eqn:hatSigma.n.m}
&& \, + \sum_{i=1}^m \left(\frac{N_n^{(i)}(c,d;a,b)}n - 
\frac{N_n(a,b)}n \cdot \frac{N_n(c,d)}n \right)
\end{eqnarray}
and
$$
N_n^{(i)}(a,b;c,d) :=\#\{j\in \{0,\ldots, n-1\}:\, 
(X_j,X_{j+1}, X_{j+i},X_{j+i+1})=(a,b,c,d)\}\,.
$$
Recall that we treat genome sequences as circular, so that $X_{n+i} = X_i$ for 
$i=0,\ldots,n-1$.

Now, from the law of large numbers for Gibbs measures,
$$
\lim\limits_{n\to \infty}\frac{N_n^{(i)}(a,b;c,d)}n =\PP([ab]_0\cap[cd]_i) \;\;\;
\PP-\hbox{a.e.}
 $$
and so
$$
\lim\limits_{n\to \infty} \hatSigma_{n,m}(a,b;c,d)=\Sigma^\varphi_{(m)}(a,b;c,d)
\;\;\; \PP-\hbox{a.e.}
$$
where
\begin{eqnarray}
\nonumber
\Sigma^\varphi_{(m)}(a,b;c,d)&=& 
\PP([ab]_0\cap[cd]_0)-\PP([ab]_0)\PP([cd]_0) 
+ \sum_{i=1}^m \left[\PP([ab]_0\cap[cd]_i) -\PP([ab]_0)\PP([cd]_0) \right] \\
\label{eqn:sigma.n.m}
&{}& \, + \sum_{i=1}^m \left[\PP([cd]_0\cap[ab]_i) -\PP([ab]_0)\PP([cd]_0) \right]\,.
\end{eqnarray}
However,
\begin{equation}
\label{eqn:sigma.def}
\Sigma^\varphi(a,b;c,d)=\lim\limits_{m\to \infty}\Sigma^\varphi_{(m)}(a,b;c,d). 
\end{equation}

Now, we claim that there exists a sequence $(m(n): n\geq 1)$ which monotonically 
increases to~$\infty$ such that
$$
\hatSigma_{n,m(n)}(a,b;c,d) \inprobabilityto{n} \Sigma^\varphi(a,b;c,d)\,,
$$
where $\inprobabilityto{n}$ is used to denote convergence in probability.  
To show this, first recall that convergence in probability is metrizable
by some metric $D$, for instance, $D(g,h)=\EE\bigl((|g-h|/(1+|g-h|)\bigr)$.  Since
$\lim\limits_{n\to \infty} D(\hatSigma_{n,m}(a,b;c,d),\Sigma^\varphi_{(m)}(a,b;c,d))=0$, 
we deduce that for all $m\geq 1$, there exists a positive integer $N(m)$ satisfying
$$
\forall n\geq N(m)\,,
\forall \, k\in \{1,\ldots, m\}:\;\;\, 
D(\hatSigma_{n,k}(a,b;c,d),\Sigma^\varphi_{(k)}(a,b;c,d))\leq 1/m\,.
$$
The sequence $(N(m): m\geq 1)$ is increasing. Now for all $n<N(1)$, 
we set $m(n)=1$ and, for $n\geq N(1)$, we  define
$m(n)=\sup\{m: N(m)\leq n\}$. By construction $m(n)$ increases with $n$. On
the other hand, since $m(n)\geq m$ for all $n\geq N(m)$, we have
$\lim\limits_{n\to \infty}m(n)=\infty$.
By construction we have
$$
\forall n\geq N(1)\,,\,:\;\;\,
D(\hatSigma_{n,m(n)}(a,b;c,d),\Sigma^\varphi_{(m(n))}(a,b;c,d))\leq 1/m(n)\,,
$$
and the claim follows by letting $m\to\infty$ and taking account of (\ref{eqn:sigma.def}).

Let $\hatV_{n,m}=
\left(\hatV_{n,m}(a,b;c,d): (a,b),(c,d)\in\K\right)$, where
$$
\hatV_{n,m}(a,b;c,d) :=
\hatSigma_{n,m}(a,b;c,d)
+\hatSigma_{n,m}(\Gamma(b), \Gamma(a) ;\Gamma(d), \Gamma(c))
-\hatSigma_{n,m}(\Gamma(b), \Gamma(a); c,d)
-\hatSigma_{n,m}(a,b;\Gamma(d),\Gamma(c)).
$$
Since $\hatSigma_{n,m(n)}$ converges in probability to
$\Sigma^\varphi$, it follows that $\hatV_n:=\hatV_{n,m(n)}$ must converge in
probability to~$V$.  Furthermore, 
$\hatV_{n}^{-1}$ 
will also converge to $V^{-1}$ in probability, provided that 
$V$ is positive definite.
To summarize, we have shown that $\hatf_n^\K$ is a consistent (unbiased)
estimator of $f^\K$ while $\hatV_n$ constitutes a consistent (but biased)
estimator of $V$.

Next, we prove the asymptotic behavior of
$\hateta_n$ claimed in Proposition \ref{prop:testasym}. 
 Recall that~$V$ is positive definite since $\Sigma^\varphi$ is positive
definite.  We have shown that $\sqrt n\hatf_n^\K$ converges in distribution to
$\N(0,V)$ and $\hatV_n^{-1}$ converges in
probability to $V^{-1}$.  This implies that 
$$ 
n \hatf_n^\K{}'
\hatV_n^{-1} \hatf_n^\K - n  \hatf_n^\K{}'
V^{-1} \hatf_n^\K = n \hatf_n^\K{}' \left(
\hatV_n^{-1} -V^{-1}\right) \hatf_n^\K \inprobabilityto{n} 0. 
$$ 
Combining this with the aforementioned fact that $n  \hatf_n^\K{}' V^{-1} 
\hatf_n^\K$ converges in distribution to a $\chi_5^2$ random variable, we see 
that $\hateta_n$ converges in distribution to a $\chi_5^2$ distribution as 
$n\to\infty$ and hence Proposition~\ref{prop:testasym} has been proved.

\subsection{Practical considerations}

When computing the statistic $\hateta_n$ for a real genomic sequence, the
parameter~$n$ is dictated by the length (in bases) of the genome under study.
However, it is necessary to choose an appropriate value for the parameter~$m$
and this is not so straight forward. 
The regime $(m(n))$ derived in the previous 
subsection is not unique. In fact, the convergence results in the preceding 
subsection remain valid for any sequence that converges 
to~$\infty$ more slowly than $(m(n))$. 
Consequently, any value $m(n) \ll n$ should satisfy the consistency criterion. 
On the other hand, the exponential cluster property of Gibbs measures implies 
that terms of the series in~(\ref{eqn:sigma.n.m}) should tend geometrically 
toward zero and the same should also be true for terms of the series in 
(\ref{eqn:hatSigma.n.m}) when~$n$ is large.  Eventually, there should come a 
point after which the terms of (\ref{eqn:hatSigma.n.m}) will constitute noise of 
the estimators and these should be ignored.  Consistency of the estimator 
$\hatSigma_{n,m}$, together with the exponentially fast  convergence of 
$\Sigma_{(m)}^\varphi$ to~$\Sigma^\varphi$, means that satisfactory results 
should be obtainable by setting $m(n)$ small relative to~$n$ when computing 
$\hateta_n$.

In our implementation of this test of CSPR for dinucleotides, 
we chose $m(n)$ to be the smallest value of~$i$ such that
$$
\abs{\frac{N_n^{(i)}(a,b,a,b)}n - \left(\frac{N_n(a,b)}n\right)^2}
\leq 0.01 \left( \frac{N_n(a,b)}n - \left(\frac{N_n(a,b)}n\right)^2\right) = 0.01 \hatVar([ab]_0)
\;\; \forall (a,b)\in\A^2.
$$
Here, $\hatVar([ab]_0)$ denotes a consistent estimator of the variance of the 
frequency of the dinucleotide $ab$ in a genome sequence.  Since $\ds 
\abs{\frac{N_n(a,b)}n - \left(\frac{N_n(a,b)}n\right)^2}$ is typically on the 
order of $0.06$, we conjecture that truncating the covariance estimators at the 
point where the sums composing the estimators change by less than $1$\% of $\hatVar([ab]_0)$ is reasonable.

Finally, some numerical experimentation leads us to conjecture that the test is 
highly powerful.  Markov chains constitute a subset of the Gibbsian processes. 
Simulating genomic sequences from Markov chains which fail to comply with CSPR 
yields a rejection rate of $100$\% at the $5$\% significance level.  Performing 
the same experiments on Markov chains that do satisfy CSPR results in a 
rejection rate close to the~$\alpha$ chosen for the test, as one would expect. 
We obtained similar results for genomic sequences generated as realizations of 
Markov random fields.  A Markov random 
field with maximal clique size~$k$ is equivalent to a Gibbs measure whose energy~$\Psi$ takes the form
$$
\Psi(x) = \sum_{j=1}^k\sum_{i=0}^{n-1}\psi^{(j)}(x_i,\ldots,x_{i+j-1}).
 $$
In other words, the energy is a linear sum of functions $\Psi^{(j)}$, each of 
which depends on cliques of size~$j$, that is, sets of~$j$ mutually neighboring 
sites.  Simulations of such sequences can be produced using the Gibbs sampler 
and~$\Psi$ will be $\og\circ\I$-invariant if ~$\Psi^{(j)}$ is $\og\circ\I$-
invariant for all $j=1,\ldots,k$.  In our numerical experiments, we simulated 
sequences from Markov random fields having maximal clique sizes of~$3$ and~$4$, 
using an energy which is not $\og\circ\I$-invariant.

\subsection{Application of the test}

Although successful tests of CSPR in genomic sequences have already been 
conducted using both empirical and rigorous methods (for instance, see 
\cite{prabhu1993,mitchell,PNAS_Albrecht,hart&martinez2011}), we would like to 
test for CSPR for $R=2$ under a Gibbsian hypothesis.

We considered a set of $1049$ complete bacterial genome sequences obtained from 
the GenBank `genomes' repository. Length and $GC$-content statistics for the set 
of genomes are shown in Table~\ref{tab:lengths.gcs}.

\begin{table}[htbp]
\caption{Summary statistics for the lengths and $GC$-contents of a collection of 
$1049$ complete bacterial genome sequence obtained from the GenBank repository.}
\label{tab:lengths.gcs}

\begin{center}
\begin{tabular}{|lrrrrr|}
\hline
Property & First Quartile & Median & Third Quartile & Mean & Std Deviation \\
\hline
Length & $1906322$ & $2976212$ & $4603746$ & $3317355$ & $1759175$ \\
$GC$-content & $0.3769$ & $0.4753$ & $0.6035$ & $0.4839$ & $0.1326$ \\
\hline
\end{tabular}
\end{center}
\end{table}

To correct for multiple testing of a large number of genomes, we used the Holm-
Bonferroni method, whose application posed no difficulties since $p$-values for 
the test statistic are readily obtainable from the $\chi_5^2$ distribution.  Of 
the $1049$ genomes tested at the $\alpha=0.01$ level of significance, the null 
was accepted in $410$ cases and was rejected in the remaining $639$ genomes.  We 
found no  relationship between $GC$-content, genome length and rejection of the 
null.

Note that the Gibbsian assumption determines the form of the covariance matrix 
$\Sigma^\varphi$, which exists as a consequence of the exponential cluster 
property.  Any genomic sequence that departs significantly from exponential 
clustering will give rise to an $\hateta_n$ far out in the tail of the 
$\chi_5^2$ distribution, since $\Sigma^\varphi$ is likely to be near singular in 
such cases.  A caveat with the test proposed here is that when a sequence is 
rejected, the reason for its rejection is unclear.  Rejection could be due to 
either violation of CSPR or lack of compliance with the Gibbsian or translation 
invariance assumptions.  In 
any case, further examination is warranted in order to discover why a particular 
sequence is rejected.  On the other hand, given the test's apparently high 
sensitivity to departures from a Gibbsian structure, sequences for which the 
null hypothesis is accepted must comply much more closely to CSPR and exhibit a 
much stronger Gibbsian-like structure than those that are rejected.

\section*{Acknowledgements}

This work was supported by the CMM Basal CONICYT Program PFB 03. We would like 
to thank researchers and engineers in the Laboratory of Bioinformatics and Mathematics of 
the Genome in the CMM for providing assistance and advice. We also want to thank 
Dr Catherine Matias from CNRS, Laboratoire Statistique et G\'enome, Universit\'e 
\'Evry, for invaluable discussions.


\begin{thebibliography}{10}

\bibitem{Chargaff}
Chargaff E.
\newblock Chemical specificity of nucleic acids and mechanism of their
  enzymatic degradation.
\newblock Experientia. 1950;6(6):201--9.

\bibitem{Chargaff2}
Rudner R, Karkas JD, Chargaff E.
\newblock Separation of {B}. subtilis {DNA} into complementary strands. {III}.
  Direct Analysis.
\newblock Proc Natl Acad Sci USA. 1968;60:921--922.

\bibitem{prabhu1993}
Prabhu VV.
\newblock Symmetry observations in long nucleotide sequences.
\newblock Nucleic Acids Res. 1993;21(12):2797--2800.

\bibitem{mitchell}
Mitchell D, Bridge R.
\newblock A test of Chargaff's second rule.
\newblock Biochem Biophys Res Commun. 2006;340(1):90--94.

\bibitem{PNAS_Albrecht}
Albrecht-Buehler G.
\newblock Asymptotically increasing compliance of genomes with Chargaff's
  second parity rules through inversions and inverted transpositions.
\newblock PNAS. 2006;103(47):17828--17833.

\bibitem{hart&martinez2011}
Hart AG, Mart\'\i~nez S.
\newblock Statistical testing of {Chargaff's} second parity rule in bacterial
  genome sequences.
\newblock Stoch Models. 2011;27(2):1--46.

\bibitem{lobry1995}
Lobry JR.
\newblock Properties of a general model of dna evolution under no-strand-bias
  conditions.
\newblock J Mol Evol. 1995;40(3):326--330.

\bibitem{sueoka1995}
Sueoka N.
\newblock ntrastrand parity rules of dna base composition and usage biases of
  synonymous codons.
\newblock J Mol Evol. 1995;40(3):18--325.

\bibitem{forsdyke&bell2004}
Forsdyke R, J BS.
\newblock Purine loading, stem-loops and Chargaff's second parity rule: A
  discussion of the application of elementary principles to early chemical
  observations.
\newblock Appl Bioinfor. 2004;3(3):3--8.

\bibitem{zhang&huang2010}
Zhang SH, Huang YZ.
\newblock Limited contribution of stem-loop potential to symmetry of
  single-stranded genomic dna.
\newblock Bioinformatics. 2010;26(4):478--485.

\bibitem{bell&forsdyke1999}
J BS, Forsdyke R.
\newblock Deviations from Chargaff's second parity rule correlate with
  direction of transcription.
\newblock J Theor Biol. 1999;197:63--76.

\bibitem{Powdel&etal2009}
Powdel BR, Satapathy SS, Kumar A, Jha PK, Buragohain AK, Borah M, et~al.
\newblock A study in entire chromosomes of violation of the intra-strand parity
  of complementary nucleotides.
\newblock DNA Research. 2009;16:325--343.

\bibitem{bowen2008}
Bowen R.
\newblock Equilibrium states and the ergodic theory of {Anosov}
  diffeomorphisms. vol. 470 of Lecture Notes in Mathematics.
\newblock revised ed. Berlin: Springer-Verlag; 2008.
\newblock With a preface by David Ruelle, Edited by Jean-Ren{\'e} Chazottes.

\bibitem{Coelho1990}
Coelho Z, Parry W.
\newblock Central limit asymptotics for shifts of finite type.
\newblock Israel Journal of Mathematics. 1990 June;69(2):235--249.
\newblock Available from:
  \url{http://www.springerlink.com/content/g415tk6310717655/}.

\end{thebibliography}

\end{document}